\documentclass{amsart}
\numberwithin{equation}{section}
\usepackage{graphicx}
\usepackage{wrapfig}
\usepackage{textcomp}
\usepackage{lmodern}
\usepackage{color}
\newtheorem{theorem}{Theorem}[section]
\newtheorem{lemma}[theorem]{Lemma}
\newtheorem{corollary}[theorem]{Corollary}
\newtheorem{proposition}[theorem]{Proposition}

\newtheorem{remark}[theorem]{Remark}

\begin{document}

\author{Ravshan Ashurov}
\address{Ashurov R: V. I. Romanovskiy Institute of Mathematics, Uzbekistan Academy of Sciences, Tashkent, 100174 Uzbekistan. \\
Central Asian University, Mirzo Ulugbek District, Tashkent, 111221 Uzbekistan.
}
\email{ashurovr@gmail.com}

\author{Masahiro Yamamoto}
\address{Yamamoto M.: Graduate School of Mathematical Sciences, The University of Tokyo, Komaba, Meguro, Tokyo 153-8914
Japan.\\
Department of Mathematics, Faculty of Arts and Sciences, Zonguldak B\"ulent Ecevit University, Zonguldak
67100 Turkey. \\
Correspondence Member of Accademia Peloritana dei Pericolanti, Palazzo Universit\`a, Piazza S.~Pugliatti 1, 98122 Messina, Italy.\\
Honorary Member of the Academy of Romanian Scientists, Ilfov, nr.~3, 050094 Bucharest, Romania.}
\email{myama@ms.u-tokyo.ac.jp}

\small

\title[Uniqueness for an inverse problem of determining order ...]{Uniqueness for an inverse problem of determining order and temporal factor of the source for time-fractional evolution equations}

\begin{abstract} 
This paper addresses the inverse problem of simultaneously recovering the fractional order $\alpha \in (0,1)\cup (1,2)$ and the time-dependent source factor $p(t)$ in the Cauchy problem for an evolution equation with a general self-adjoint operator $A$ in a Hilbert space $X$. The overdetermination condition is given by the scalar product $( u(t), \psi)_X$ for $0 < t < T$, where $\psi \in D(A)$ is an arbitrary fixed element.  

Uniqueness of the fractional order $\alpha$ is established independently of the specific form of the elliptic operator $A$ and the source function $p(t)$. Furthermore, uniqueness of the factor $p(t)$ is proved not only under the trivial overdetermination $( u(t), \psi)_X = 0$ for all $t \in (0,T)$, but also when the function $t \mapsto ( u(t), \psi)_X$ possesses sufficient smoothness.  

The proof relies on a decomposition of the solution near $t=0$ into a least smooth component and a smoother remainder.

\end{abstract}

\maketitle

\section{Introduction and main result}
Let $X$ be a Hilbert space with the scalar product $(\cdot, \cdot)_{X}$ and the norm $\|\cdot\|_{X}$. Let $A$ be a positive definite self-adjoint operator with the compact inverse. We denote the domain of the operator $A$ by $\mathcal{D}(A)$.

We define a fractional time derivative in a Hilbert space $X$ (e.g., Yamamoto \cite{7}, \cite{8}), which provides a convenient framework and is equivalent to the classical Caputo derivative if functions under consideration are smooth. 

Let $0<\alpha<1$. For vector-valued functions, fractional Riemann-Liouville integrals are defined in the same way as for scalar functions (the integral is understood in the sense of the Bochner integral; see, e.g., \cite{Liz})
$$
J^{\alpha} w(t):=\frac{1}{\Gamma(\alpha)} \int_{0}^{t}(t-s)^{\alpha-1} w(s) d s
$$
for $w \in L^{2}(0, T ; X)$. Then, we can readily prove that $J^{\alpha}: L^{2}(0, T ; X) \longrightarrow L^{2}(0, T ; X)$ is injective (see, e.g., \cite{2}). Therefore, we can define
$$
\partial_{t}^{\alpha}:=\left(J^{\alpha}\right)^{-1}, \quad \mathcal{D}\left(\partial_{t}^{\alpha}\right)=J^{\alpha} L^{2}(0, T ; X) .
$$
Then, it is easy that we define a Banach space $H_{\alpha}(0, T ; X):=J^{\alpha} L^{2}(0, T ; X)$ with the norm $\|v\|_{H_{\alpha}(0, T ; X)}:=\left\|\left(J^{\alpha}\right)^{-1} v\right\|_{L^{2}(0, T ; X)}$. Moreover, we can readily verify that $\partial_{t}^{\alpha}$ is a closed operator defined in $H_{\alpha}(0, T ; X)$. Then, we can see (see Kubica, Ryszewska and  Yamamoto \cite{4}, Chapters 2, when $X= \mathbb{R}$; see Chapter 4 of \cite{4} and Theorem 2.4.1 
in Yamamoto \cite{8} for a Hilbert space $X$):
\begin{equation}\label{Halpha}
H_{\alpha}(0, T ; X)=\left\{
\begin{aligned}
&H^{\alpha}(0, T ; X),
&& 0 < \alpha < \frac{1}{2}, \\
& \bigg\{v \in H^{\frac{1}{2}}(0, T ; X); \, \int_0^T \frac{||v(t)||_X^2}{t} dt <\infty \bigg\},
&&\alpha = \frac{1}{2}, \\
& \left\{v \in H^{\alpha}(0, T ; X) ; \,\,v(0)=0\right\}, 
&& \frac{1}{2}<\alpha<1.
\end{aligned}
\right.
\end{equation}
Here $H^{\alpha}(0, T ; X)$ is a Sobolev-Slobodeckij space (e.g., Adams \cite{1}, Chapter 7), and the norm $\|v\|_{H^{\alpha}(0, T ; X)}$ is equivalent to
$$
\left(\|v\|^2_{L^{2}(0, T ; X)}+\int_{0}^{T} \int_{0}^{T} \frac{\|v(t)-v(s)\|_{X}^{2}}{|t-s|^{2 \alpha+1}} d t d s\right)^{\frac{1}{2}}.
$$
Consequently, functions from $H_{\alpha}(0, T ; X)$, $\frac{1}{2} < \alpha <1$, according to the Sobolev embedding theorem (e.g., \cite{1}, Chapter 5), are continuous with respect to $t$ and vanish at the point $t=0$, while functions from $H_{\alpha}(0, T ; X)$ at $0 < \alpha<\frac{1}{2}$,  are generally not continuous and therefore one cannot speak of the value of the function at zero: $t=0$.

Now let $\alpha = 1+ \sigma$, $\sigma\in (0,1)$, and set $H_{1}(0, T ; X) = \left\{v \in H^{1}(0, T ; X) ; \,\,v(0)=0\right\}$. Then we define
\begin{equation}\label{Hbeta}
    H_{\alpha}(0, T ; X) = H_{1+\sigma}(0, T ; X) = \left\{v \in H_{1}(0, T ; X) ; \,\,\frac{d v}{d t}\in H_\sigma(0, T ; X) \right\},
\end{equation}
the derivative is understood in a strong sense. This space was introduced in Yamamoto \cite{7}, Chapter 2 for $X= \mathbb{R}$ and in Chapter 6 for $X= L_2$.
In the same paper it is proved that $J^{1+\sigma}: L^{2}(0, T ; X) \longrightarrow H_{1+\sigma}(0, T ; X)$ is an isomorphism. With this in mind, the author of the work \cite{7} introduced the following definitions:
\[
\partial_t^{\alpha}  = \partial_t^{1+\sigma} = (J^{1+\sigma})^{-1},\,\, \mathcal{D}\left(\partial_{t}^{1+\sigma}\right) = H_{1+\sigma} (0, T; X).
\]

Now we are ready to formulate our inverse problems and main results.

Let $\alpha\in (0,1) \cup (1,2)$. For fixed $f \in X$, we consider the solution $u_{p, \alpha}$ to
\begin{equation}\label{sol}
\partial_{t}^{\alpha} u(t)=-A u(t)+p(t) f \quad \text { in } X \times(0, \infty), \quad u\in H_{\alpha}(0, T ; X).
\end{equation}
Let us emphasize that the inclusion $u\in H_{\alpha}(0, T ; X)$ here replaces the initial condition. If $\frac{1}{2}<\alpha<1$, then we indeed have the Cauchy condition:  $u(0)=0$ (see (\ref{Halpha}). Similarly, if $\alpha= 1+ \sigma$, and $\frac{1}{2}<\sigma<1$, then $u(0)=0$ and $u'(0)=0$ (see (\ref{Hbeta})).
For other cases of $\alpha$, we cannot define $u(0)$ and/or $u'(0) = 0$ 
themselves in the sense of the trace at $t=0$, but the unique existence 
of the solution $u$ to \eqref{sol} is proved.
 (e.g., Huang and Yamamoto \cite{HY}). 
Indeed, it is known  (e.g., \cite{7}, \cite{8}) that \eqref{sol} possesses a 
unique solution $u_{p, \alpha} \in H_{\alpha}(0, T ; X) \cap 
L^{2}(0, T ; \mathcal{D}(A))$. 

Let us call \eqref{sol} the forward problem.
Our primary goal in this paper is to investigate the inverse problem of 
simultaneously determining both the source function $p(t)$ and the exponent of 
a fractional derivative $\alpha$. 

We will consider several different inverse problems, which can be summarized as follows:

\textbf{Inverse  problem:}

Let $\psi \in X$ be fixed. Determine $\alpha\in (0,1) \cup (1,2)$ and $p \in 
C[0, T]$ by data $(u(t), \psi)_{X}$ for $0<t<T$.

Numerous works by various authors have been devoted to the study of such 
problems. 
A detailed overview of the most significant contributions is presented in the 
survey papers 
by Li, Liu, and Yamamoto~\cite{LiYamamoto1} and by Li and Yamamoto~\cite{LiYamamoto2}. 
An analysis of the results reported in~\cite{LiYamamoto1}--\cite{LiYamamoto2} shows that, 
in most of these works, subdiffusion equations are considered, and the unknown parameter 
$\alpha$ is recovered from the observation data $u(x_0,t)$, $0 < t \leq T$, where $x_0$ 
is a fixed point in the domain $\Omega$ in which the equation is posed.

Let now $\psi \in X$ be a weight function. We then ask whether the fractional order 
$\alpha$ can be determined from the additional information
\begin{equation}\label{overdetermination}
    (u(t_0), \psi)_X = d_0,
\end{equation}
for some fixed time $t_0$ and constant $d_0$. 
This scalar product can be interpreted as a weighted average of the solution 
$u(t)$ at time $t_0$.
We emphasize that data \eqref{overdetermination} 
are the minimum for the inverse problem, because
an unknown is one number $\alpha$, and so the uniqueness result with the 
data is the best possible.   
Indeed, in~\cite{AU2020}, an affirmative answer to this question is obtained in the case where 
the weight function is the first eigenfunction of the operator $A$. 
Moreover, in~\cite{PskhuUzmat}, 
an affirmative result is established for the case where $\psi(x)$ is an arbitrary harmonic 
function (i.e., $\Delta \psi(x) = 0$ and $A = -\Delta$) that satisfies the Tikhonov-type condition; in that work, the problem is considered in the whole space $\mathbb{R}^N$.

The present paper shows that, in condition~\eqref{overdetermination}, an arbitrary element 
$\psi \in D(A)$ can be used as a weight function, provided that the condition is imposed 
not at a fixed time $t=t_0$, but for all $t \in (0,T)$. Furthermore, we prove the uniqueness 
not only of the fractional order $\alpha$ but also of the time-dependent factor $p(t)$ in 
the source term.

Let us formulate our main result. For this, we will look at another system
\begin{equation}\label{sol11}
\partial_{t}^{\beta} \widetilde{u}(t)=-\widetilde{A} \widetilde{u}(t)+q(t) f \quad \text { in } X \times(0, \infty), \quad  \widetilde{u}\in H_{\beta}(0, T ; X), \quad \beta\in (0, 1)\cup (1,2),
\end{equation}
where $\widetilde{A}$ is a positive definite self-adjoint operator with the compact inverse. 
Then there exists a unique solution $\widetilde{u}_{q, \beta} \in H_{\beta}(0, T ; X) \cap L^{2}(0, T ; \mathcal{D}(\widetilde{A}))$.

\textbf{Inverse Problem 1.} 
Does the equality
\[
(u_{p,\alpha}(t), \psi)_X = (u_{q,\beta}(t), \widetilde{\psi})_X, \quad 0 < t < T,
\]
for some \(\psi, \widetilde{\psi} \in X\), imply that \(\alpha = \beta\) and \(p = q\) on \((0,T)\)?

To study this question, we introduce the following admissible class of unknown functions \(p\) and \(q\):

$$
\mathcal{P} := \Bigl\{ p \in C[0,T] : \text{there exist } C = C(p) > 0,\ 
m = m(p) \in \mathbb{Z}_{\geq 0},\ a = a(p) > 0,
$$
$$ \text{and}\ p_0 = p_0(p) \neq 0\text{ such that } a - m \geq \varepsilon > 0 \text{ and }
$$
\begin{equation}
\bigl| p(t) - p_0 t^m \bigr| \leq C t^a \quad \text{for all}\ 0 \leq t \leq T \Bigr\}.
\end{equation}

Here and henceforth we write
$\mathbb{Z}_{\geq 0}:= \{ m\in \mathbb{Z};\, m\ge 0\}$.

A typical example of this class is
\[
\mathcal{P}_1 := \Bigl\{ p \in C^\ell[0,T] \text{ for some } \ell = \ell(p) \in \mathbb{N} : 
\text{there exists } \ell_0 = \ell_0(p) \in \{0,1,\dots,\ell-1\}
\]
\[
\text{ such that } \frac{d^{\ell_0} p}{dt^{\ell_0}}(0) \neq 0 \Bigr\}.
\]

In this case \(a = m + 1\) and \(\varepsilon = 1\).

\begin{theorem}\label{thm1}
Let \(p, q \in \mathcal{P}\), \(\psi \in \mathcal{D}(A)\), \(\widetilde{\psi} \in \mathcal{D}(\widetilde{A})\), and
\[
(f,\psi)_X \neq 0, \qquad (\widetilde{f},\widetilde{\psi})_X \neq 0.
\]
If
\[
(u_{p,\alpha}(t), \psi)_X = (u_{q,\beta}(t), \widetilde{\psi})_X, \quad 0 < t < T,
\]
and both parameters \(\alpha\) and \(\beta\) belong either to \((0,1)\) or to \((1,2)\), then \(\alpha = \beta\).

Furthermore, assume that \(A = \widetilde{A}\), \((f,\psi)_X \neq 0\), and \(\psi \in \mathcal{D}(A)\). If
\[
(u_{p,\alpha}(t), \psi)_X = (u_{q,\beta}(t), \psi)_X, \quad 0 < t < T,
\]
and \(p,q \in \mathcal{P}\), then \(\alpha = \beta\) and \(p = q\) on \((0,T)\).
\end{theorem}

We now introduce a narrower admissible class which allows us to remove the a priori restriction on the interval containing \(\alpha\) and \(\beta\):

\[
\widetilde{\mathcal{P}} := \Bigl\{ p \in C[0,T] : \text{there exist } C = C(p) > 0,\ 
m = m(p) \in \mathbb{Z}_{\geq 0},\ a = a(p) > 0, \]
$$
\text{ and }\ p_0 \neq 0
\text{ such that } a \geq m + 2 \text{ and }$$
\begin{equation}
\bigl| p(t) - p_0 t^m \bigr| \leq C t^a \quad \text{ for all } \ 0 \leq t \leq T \Bigr\}.
\end{equation}

An example of such a class is
\[
\widetilde{\mathcal{P}}_1 := \Bigl\{ p \in C^\ell[0,T] : \text{ there exists } \ell_0 = \ell_0(p) \in \{0,1,\dots,\ell-1\}
\]
$$ \text{ such that } \frac{d^{\ell_0} p}{dt^{\ell_0}}(0) \neq 0 \text{ and } \frac{d^{\ell_0+1} p}{dt^{\ell_0+1}}(0) = 0 \Bigr\}.$$

In this case \(a \geq m + 2\).

\begin{theorem}\label{thm11}
Let \(p, q \in \widetilde{\mathcal{P}}\), \(\psi \in \mathcal{D}(A^2)\), \(\widetilde{\psi} \in \mathcal{D}(\widetilde{A^2})\), and
\[
(f,\psi)_X \neq 0, \qquad (\widetilde{f},\widetilde{\psi})_X \neq 0.
\]
Then the equality
\[
(u_{p,\alpha}(t), \psi)_X = (u_{q,\beta}(t), \widetilde{\psi})_X, \quad 0 < t < T,
\]
implies \(\alpha = \beta\) even if \(\alpha, \beta \in (0,1) \cup (1,2)\).
\end{theorem}

Thus, the theorem asserts that if its conditions are met - note that 
in this case, condition \(\psi \in \mathcal{D}(A^2)\), \(\widetilde{\psi} \in \mathcal{D}(\widetilde{A^2})\) applies -then, even if we do not know a priori in which of the sets $(0, 1)$ or $(1, 2)$ each of the parameters $\alpha$ and $\beta$ lies, they are nonetheless identical.

The uniqueness in determining the order with the assumption $A=\widetilde{A}$ is easily proved (Miller and Yamamoto, \cite{MilYam}, for example). Our main result, Theorems \ref{thm1} and \ref{thm11} assert that the order $\alpha$ is uniquely determined even if we do not know the operator $A$ and the source term.

\textbf{Inverse Problem 2.} Find a solution to the Cauchy problem
\begin{equation}\label{sol1}
\partial_{t}^{\alpha} u(t)=-A u(t)+p(t) f \quad \text { in } X \times(0, \infty), \quad u\in H_{\alpha}(0, T ; X),
\end{equation}
and the source factor $p(t)\in C[0, T]$ using the overdetermination condition 
\begin{equation}\label{overdef1}
\left(u(t), \psi\right)_{X}(t) = \Psi(t), \,\, 0<t<T.
\end{equation}

\begin{theorem}\label{thm3}
Let $\alpha \in (0,1)\cup(1,2)$. Assume that $\psi \in \mathcal{D}(A)$ and 
$(f,\psi)_X \neq 0$. If $\partial_t^\alpha \Psi(t) \in C[0,T]$, then there exists a unique 
solution $\{u,p\}$ of Inverse Problem~2, (\ref{sol1})--(\ref{overdef1}). Moreover, the following two-sided Lipschitz stability estimate holds
\begin{equation}\label{stability}
C^{-1} \|\partial_t^\alpha \Psi\|_{C[0,T]} \leq \|p\|_{C[0,T]} \leq C \|\partial_t^\alpha \Psi\|_{C[0,T]},
\end{equation}
where the constant $C>0$ depends on $f$, $\psi$, and $\alpha$.
\end{theorem}
In the case $\alpha \in (0,1)$, this theorem can be regarded as an abstract version of Theorem~1 in~\cite{LiYamamoto2}.

{\color{blue} All of these theorems state that if 
\[
(f,\psi)_X \neq 0,\qquad  \text{and} \qquad (u_{p,\alpha}(t), \psi)_X = 0,
\]
then 
\[
p(t) = 0.
\] 
In the proof of these theorems, the fact that the initial function is zero plays a significant role:
\[
u\in H_{\alpha}(0, T ; X).
\]

The question naturally arises: if the initial data is nonzero: $u(0) \neq 0$, will the assertions of these theorems be true? In other words, are the statements of the formulated theorems true for the Cauchy problem:

Let $a\neq 0\, (\in X)$ and $\alpha\in (0,1)$;
\[
\partial_{t}^{\alpha} (u(t) - a)=-A u(t)+p(t) f \quad \text { in } X \times(0, \infty), \quad u - a\in H_{\alpha}(0, T ; X).
\]

The following example shows that this is not the case. 

\textbf{Example}. We consider
$$
\left\{\begin{array}{rl}
&  D_t^\alpha u - u_{xx} = p(t)f(x), \quad x\in (0,1), \, t>0,  \,\, \alpha\in (0,1),\\
& u(x, 0) = a(x) \quad \mbox{in $ (0,1)$}, \\
& u(0,t) = u(1,t) = 0.
\end{array}\right.
$$
We will look for a solution in the form
\[
u(x,t) = E_\alpha(t^\alpha) g(x).
\]
Note $D_t^\alpha [ E_\alpha(t^\alpha) ] = E_\alpha(t^\alpha)$. Then
\[
p (t) =  E_\alpha(t^\alpha),\,\, f(x) = g(x) - g''(x), \,\, a(x) = g(x).
\]
We construct functions \( g, \psi \in C^2[0,1] \) such that
\[
g(0)=g(1)=0, \quad \text{to guarantee}\quad u(0,t) = u(1,t) = 0,
\]
and
\[
\int_0^1 g(x)\psi(x)\,dx = 0, \quad \int_0^1 g''(x)\psi(x)\,dx \neq 0,
\]
here the first equality gives $(u_{p,\alpha}(t), \psi)_{L_2(0,1)} = 0$ and both equalities together provide the relation $(f,\psi)_{L_2(0,1)} \neq 0$. Note that $p (t) =  E_\alpha(t^\alpha) \neq 0$.

Thus, it suffices to construct the functions $g(x)$ and $\psi(x)$ with the specified properties. It is not hard to verify that:
\[
g(x)=x(1-x), \quad \psi(x)=1-5x(1-x)
\]
Indeed let
\[
g(x) = x(1-x), \quad x \in [0,1].
\]
Then clearly
\[
g(0)=g(1)=0, \qquad g''(x) = -2.
\]
We look for a function \( \psi \) of the form
\[
\psi(x) = 1 - c g(x),
\]
where the constant \( c \) is chosen to satisfy the orthogonality condition
\[
\int_0^1 g(x)\psi(x)dx = 0.
\]
Substituting \( \psi\), we obtain
\[
\int_0^1 g(x) dx - c \int_0^1 g^2(x) dx = 0,
\]
hence
\[
c = \frac{\int_0^1 g(x) dx}{\int_0^1 g^2(x) dx} =5.
\]
Thus, 
\[
\psi(x) = 1 - 5x(1-x).
\]
Consequently, if the initial function is not equal to zero, then the uniqueness of the function $p(t)$ may be absent.}

{\color{red} Let us now consider the case $\alpha\in (0,1)$ and set}
$$
\mathcal{P}_0:=\{p \in C[0, T] ; \,\,p(0) \neq 0\} .
$$
Then from Theorem \ref{thm1} the statement easily follows:\begin{corollary}\label{cor1}
If $A=\widetilde{A},\,\,(f, \psi)_{X} \neq 0$, $\psi \in \mathcal{D}(A)$,  and  if $\left(u_{p, \alpha}(t), \psi\right)_{X}=\left(u_{q, \beta}(t), \psi\right)_{X}$ for $0<t<T$ and $p, q \in \mathcal{P}_0$, then $\alpha = \beta$, and $p=q$ in $(0, T)$ .
    \end{corollary}
Corollary \ref{cor1}, in particular, means the uniqueness in determining $p(t)$ with given order $\alpha \in(0,1)$. More precisely, $\left(u_{p, \alpha}(t), \psi\right)_{X}=0$ for $0<t<T$ implies $p(t)=0$ for $0<t<T$. However, if $p \in \mathcal{P}_0$, then for concluding $p(t)=0$ for $0<t<T$, we do not need the zero output: $\left(u_{p, \alpha}(t), \psi\right)_{X}=0$ for $0<t<T$, but some extra regularity is sufficient for $\frac{1}{2}<\alpha<1$ :

\begin{theorem}\label{thm5}
    We assume that $\psi \in \mathcal{D}(A)$ and $(f, \psi)_{X} \neq 0$. Let $\frac{1}{2}<\alpha, \theta<1$. If $p \in \mathcal{P}_0$ and $\partial_{t}^{\alpha}\left(u_{p, \alpha}, \psi\right)_{X} \in H_{\theta}(0, T ; X)$, then $p(t)=0$ for $0<t<T$.
\end{theorem}

Note that under the conditions of  Theorem \ref{thm5} one has $u_{p, \alpha}(t)|_{t=0}=0$ and $\partial_{t}^{\alpha}\left(u_{p, \alpha}, \psi\right)_{X}|_{t=0} =0$ (see the definitions (\ref{Halpha})). It turns out that these conditions alone are sufficient for the truth of the statement of this theorem.

\begin{theorem}\label{thm6}
Assume that $\psi \in \mathcal{D}(A)$ and $(f,\psi)_X \neq 0$. 
Suppose that the functions $u_{p,\alpha}(t)$ and 
$\partial_t^\alpha \bigl(u_{p,\alpha},\psi\bigr)_X(t)$ are continuous at $t=0$, and
\[
\lim_{t\to 0^+} u_{p,\alpha}(t) = 0, 
\qquad 
\lim_{t\to 0^+} \partial_t^\alpha \bigl(u_{p,\alpha},\psi\bigr)_X(t) = 0.
\]
If $p \in \mathcal{P}_0$, then $p(t) = 0$ for all $0 < t < T$.
\end{theorem}
Note that Theorem \ref{thm6} is valid for arbitrary $\alpha\in (0, 1)$. We emphasize that if $\alpha\in (0, 1/2]$, then functions from class $H_{\alpha}(0, T ; X)$ are generally not defined at every point $t$. If $\frac{1}{2}<\alpha<1$, then condition $\lim\limits_{t\to +0} u_{p, \alpha}(t) =0$ is fulfilled automatically, since $u_{p, \alpha}(t)$ is continuous and $u_{p, \alpha}(0)=0$ (see (\ref{Halpha})).

Theorem \ref{thm5}   directly implies

\begin{corollary}\label{cor2}
Assume that $\psi \in \mathcal{D}(A)$ and $(f,\psi)_X \neq 0$. 
Let $\frac{1}{2} < \alpha, \theta < 1$. If
\[
\partial_t^\alpha \bigl(u_{p,\alpha}, \psi\bigr)_X \in H_\theta(0,T),
\]
and $p$ is a piecewise constant function on $(0,T)$, then $p(t) = 0$ for all $0 < t < T$.
\end{corollary}
Here by piecewise constant $p(t)$, we understand that we can choose $t_{0}:=0<t_{1}<t_{2}<\cdots< t_{N-1}<t_{N}:=T$ such that $\left.p\right|_{\left(t_{k}, t_{k+1}\right)}$ is constant for $0 \leq k \leq N-1$.

By Kubica, Ryszewska and Yamamoto \cite{4}, we see that if $f \in X$ and $p \in L^{2}(0, T)$, then $u_{p, \alpha} \in H_{\alpha}(0, T ; X) \cap L^{2}(0, T ; \mathcal{D}(A))$. Hence, $p \in L^{2}(0, T)$ implies $\left(u_{p, \alpha}, \psi\right)_{X} \in H_{\alpha}(0, T)$. Moreover, in view of Lemma \ref{lem3} (iii)  shown in Section \ref{sec2}, we can derive that $\partial_{t}^{\alpha}\left(u_{p, \alpha}, \psi\right)_{X} \in L^{2}(0, T)$ implies $(f, \psi)_{X} p \in L^{2}(0, T)$. Therefore, under condition $(f, \psi)_{X} \neq 0$, we see that $\left(u_{p, \alpha}, \psi\right)_{X} \in H_{\alpha}(0, T)$ if and only if $p \in L^{2}(0, T)$. Theorem \ref{thm5} means that the regularity $\left(u_{p, \alpha}, \psi\right)_{X} \in H_{\alpha+\theta}(0, T)$ with $\theta>\frac{1}{2}$ is specially strong, which means $p(t) \equiv 0$ in ( $0, T$ ). On the other hand, such strong smoothing effect for $\partial_{t}^{\alpha}\left(u_{p, \alpha}, \psi\right)_{X}$ appears always if $(f, \psi)_{X}=0$.

\begin{proposition}\label{prop1}
   Assume that $f \in X,\, (f, \psi)_{X}=0$ and $\psi \in \mathcal{D}(A)$. Then, for any $p \in L^{2}(0, T)$ and any $\varepsilon>0$, we have
$$
\left(u_{p, \alpha}, \psi\right)_{X} \in H_{2 \alpha-\varepsilon}(0, T) .
$$
\end{proposition}

This article consists of six sections. The proof is based on decomposing the solution in a neighborhood of $t=0$ into a least smooth component and a smoother component. This is precisely what is accomplished in Section \ref{sec2} through the key Lemmas \ref{lem31}, \ref{lem3}, and \ref{lem311}. Section \ref{sec3} is devoted to proving the main Theorems \ref{thm1} and \ref{thm11}. Section \ref{sec4} examines a special case involving $\alpha\in (0,1)$ and $p\in \mathcal{P}_0$ and establishes more refined results (Theorems \ref{thm5} and \ref{thm6}, Corollary \ref{cor1}, and Proposition \ref{prop1}). Section \ref{sec5} discusses a stability result regarding the determination of the time-dependent factor of the source function. Section \ref{sec6} contains concluding remarks.

\section{Key lemmata}\label{sec2}
Let $\left\{\lambda_{n}\right\}_{n \in \mathbb{N}}$ be the set of all the eigenvalues of $A$ :
$$
0<\lambda_{1}<\lambda_{2}<\cdots \longrightarrow \infty
$$
By $P_{n}, n \in \mathbb{N}$, we denote the orthogonal projection to $\operatorname{Ker}\left(A-\lambda_{n} I\right)$ with $n \in \mathbb{N}$, where $I$ is the identity operator in $X$. Then, it is well-known that $P_{n}$ is a bounded linear operator defined over $X$, and is self-adjoint,

$$
P_{n} P_{m}=\left\{\begin{array}{l}
P_{n}, \quad n=m, \\
0, \quad n \neq m,
\end{array}\right.
$$
and

\begin{equation}\label{eq2.1}
a=\sum_{n=1}^{\infty} P_{n} a \quad \text { in } X, \quad\|a\|_{X}^{2}=\sum_{n=1}^{\infty}\left\|P_{n} a\right\|_{X}^{2} \quad \text { for all } a \in X .
\end{equation}
For other positive-definite self-adjoint operator $\widetilde{A}$ with the compact inverse, we denote the set of eigenvalues $\widetilde{\lambda}_{n}$ and the projections $\widetilde{P}_{n}$ for $n \in \mathbb{N}$.

We define the Mittag-Leffler function by

\begin{equation}\label{eq2.2}
E_{\alpha, \theta}(z)=\sum_{k=0}^{\infty} \frac{z^{k}}{\Gamma(\alpha k+\theta)}, \quad z \in \mathbb{C}, \quad \alpha, \theta>0,
\end{equation}
(e.g., Podlubny [5], p. 17). Moreover, we define the convolution:
$$
(g * h)(t)=\int_{0}^{t} g(t-s) h(s) d s \quad \text { for } g, h \in L^{1}(0, T ; X).
$$

We can prove

\begin{lemma}\label{lem1}
Let $\gamma>0$ and $g, h \in L^{1}(0, T ; X)$. We assume that there exists $G \in L^{1}(0, T ; X)$ such that $g=J^{\gamma} G$. Then, $\left(J^{\gamma}\right)^{-1}(g * h) \in L^{1}(0, T ; X)$ and $\left.\left(J^{\gamma}\right)^{-1}(g * h)=\left(\left(J^{\gamma}\right)^{-1} g\right) * h\right)$.
\end{lemma}
\begin{proof}
    We can directly verify $J^{\gamma}(G * h)=\left(\left(J^{\gamma} G\right) * h\right)$. Therefore, in terms of $g=J^{\gamma} G$, we obtain
$$
J^{\gamma}(G * h)=\left(\left(J^{\gamma} G\right) * h\right)=g * h,
$$
which means $\left(J^{\gamma}\right)^{-1}(g * h)=G * h$, that is, the substitution of $G=\left(J^{\gamma}\right)^{-1} g$ yields
$$
\left(J^{\gamma}\right)^{-1}(g * h)=\left(\left(J^{\gamma}\right)^{-1} g * h\right).
$$

Thus, the proof of Lemma \ref{lem1} is complete.
\end{proof}
The scalar product $ (\cdot, f)_X : X \to \mathbb{R} $ (or $ \mathbb{C} $) is a continuous linear functional on $ X $, because 
$$
|(v, f)_X| \leq \|v\|_X \cdot \|f\|_X,\,\, v, f \in X
$$ 
(by the Cauchy-Schwarz inequality). Therefore, this functional is continuous in the norm of $X$. This means that it preserves limits, i.e., by writing the integral sum inside the scalar product, we can pass to the limit. Then
\[
J^{\gamma}(g, h)_X=\left(\left(J^{\gamma} g\right),  h\right)_X,
\]
where $\gamma>0$ and $g, h \in L^{2}(0, T ; X)$.

Taking this relation into account, Lemma \ref{lem1} can be proven in a completely similar way by taking the scalar product instead of the convolution, i.e., the following statement is true
\begin{lemma}\label{lem1.1}
Let $\gamma>0$ and $g, h \in L^{2}(0, T ; X)$. We assume that there exists $G \in L^{2}(0, T ; X)$ such that $g=J^{\gamma} G$. Then, $\left(J^{\gamma}\right)^{-1}(g,  h)_X \in L^{2}(0, T ; X)$ and $\left.\left(J^{\gamma}\right)^{-1}(g, h)_X=\left(\left(J^{\gamma}\right)^{-1} g\right), h\right)_X$.
\end{lemma}

We have the eigenfunction expansion of the solution (see, e.g., \cite{6}, \cite{8}).

\begin{lemma}\label{lem2}
    Let $p \in L^{2}(0, T)$ and $f \in X$. Then, the solution of Cauchy problem (\ref{sol}) has the form 
\begin{equation}\label{eq2.3}
u_{p, \alpha}(t)=\sum_{n=1}^{\infty}\left(t^{\alpha-1} E_{\alpha, \alpha}\left(-\lambda_{n} t^{\alpha}\right) * p\right)(t) P_{n} f, \quad 0<t<T,
\end{equation}
where the series is convergent in $H_{\alpha}(0, T; X) \cap L^{2}(0, T ; \mathcal{D}(A))$.
\end{lemma}
In view of Lemma \ref{lem2}, we can prove

\begin{lemma}\label{lem31}
    Let $\psi \in \mathcal{D}(A)$, $f \in X$, and $\alpha\in (0,1) \cup (1,2)$. Then,
  $$
\begin{aligned}
& \left(u_{p, \alpha}(t), \psi\right)_{X}=(f, \psi)_{X}\left(J^{\alpha} p\right)(t)-\sum_{n=1}^{\infty}\left(t^{2 \alpha-1} E_{\alpha, 2 \alpha}\left(-\lambda_{n} t^{\alpha}\right) * p\right)(t)\left(\lambda_{n} P_{n} f, \psi\right)_{X} \\
= & :(f, \psi)_{X}\left(J^{\alpha} p\right)(t)+R_{p, \alpha}(t),
\end{aligned}
$$
and
\begin{equation}\label{R}
   |R_{p, \alpha}(t)| \leq C t^{2\alpha} \|p\|_{L^{\infty}(0, T)}\|f\|_{X}\|A \psi\|_{X}, \,\, 0\leq t < 1.
\end{equation}
\end{lemma}
\begin{proof}
    To decompose the power series \eqref{eq2.3} into the first and the second parts, we use that
\begin{equation}\label{eq2.4}
E_{\alpha, \alpha}(z)=\frac{1}{\Gamma(\alpha)}+z E_{\alpha, 2 \alpha}(z), \quad z \in \mathbb{C}, 
\end{equation}
which can be derived directly from \eqref{eq2.2}. 

Substitution of \eqref{eq2.4} into \eqref{eq2.3} yields
\begin{equation}\label{eq2.5}
u_{p, \alpha}(t)=\sum_{n=1}^{\infty}\left(t^{\alpha-1} \frac{1}{\Gamma(\alpha)} * p\right)(t) P_{n} f-\sum_{n=1}^{\infty}\left(t^{2 \alpha-1} E_{\alpha, 2 \alpha}\left(-\lambda_{n} t^{\alpha}\right) * p\right)(t) \lambda_{n} P_{n} f.
\end{equation}
Here, the first term of the right-hand side has the form by \eqref{eq2.1}
$$
\left(\frac{t^{\alpha-1}}{\Gamma(\alpha)} * p\right)(t) \sum_{n=1}^{\infty} P_{n} f=\left(\frac{t^{\alpha-1}}{\Gamma(\alpha)} * p\right)(t) f=\left(J^{\alpha} p\right)(t) f.
$$
Therefore,
\begin{equation}\label{eq2.6}
u_{p, \alpha}(t)=\left(J^{\alpha} p\right)(t) f-\sum_{n=1}^{\infty}\left(t^{2 \alpha-1} E_{\alpha, 2 \alpha}\left(-\lambda_{n} t^{\alpha}\right) * p\right)(t) \lambda_{n} P_{n} f .
\end{equation}
Taking the scalar product with $\psi \in \mathcal{D}(A)$, we obtain the first statement of Lemma \ref{lem31}.

To prove estimate~\eqref{R}, recall that for all \(\alpha \in (0,2)\) and any real number \(\beta\), the Mittag-Leffler function satisfies the bound (see, e.g., Theorem~1.6, p.~35 in~\cite{5})
\begin{equation}\label{estimateML}
    |E_{\alpha,\beta}(-t)| \leq \frac{C}{1+t}, \qquad t > 0.
\end{equation}
Applying this estimate, we obtain
\[
\left| \sum_{n=1}^{\infty} \Bigl( t^{2\alpha-1} E_{\alpha,2\alpha}(-\lambda_n t^\alpha) * p \Bigr)(t) \, (\lambda_n P_n f, \psi)_X \right|
\leq C \|p\|_{L^\infty(0,T)} \|f\|_X \|A\psi\|_X \int_0^t \tau^{2\alpha-1}\, d\tau,
\]
which coincides with~\eqref{R}.
\end{proof}

\begin{lemma}\label{lem3}
    Let $\psi \in \mathcal{D}(A)$ and $f \in X$, and $0<\gamma \leq \alpha<1$. Then,\\
(i)$$
\left(\partial_{t}^{\gamma} u_{p, \alpha}(t), \psi\right)_{X}=(f, \psi)_{X}\left(J^{\alpha-\gamma} p\right)(t)-\sum_{n=1}^{\infty}\left(t^{2 \alpha-\gamma-1} E_{\alpha, 2 \alpha-\gamma}\left(-\lambda_{n} t^{\alpha}\right) * p\right)(t)\left(\lambda_{n} P_{n} f, \psi\right)_{X}
$$

in $L^{2}(0, T)$.\\
(ii)

$$
\left|\sum_{n=1}^{\infty}(t-s)^{2 \alpha-\gamma-1} E_{\alpha, 2 \alpha-\gamma}\left(-\lambda_{n}(t-s)^{\alpha}\right)\left(\lambda_{n} P_{n} f, \psi\right)_{X}\right| \leq C(t-s)^{2 \alpha-\gamma-1}\|f\|_{X}\|A \psi\|_{X}
$$

for $0<s<t<T$.\\
(iii) Let $0<\gamma<2 \alpha$. Then, (see Lemma \ref{lem31})

$$
\left\|R_{p, \alpha}\right\|_{H_{\gamma}(0, T)} \leq C\|p\|_{L^{\infty}(0, T)}\|f\|_{X}\|A \psi\|_{X}.
$$
\end{lemma}

\textbf{Proof of Lemma \ref{lem3} (i)}.

Since
$$
\left(J^{\gamma}\left(t^{2 \alpha-1-\gamma} E_{\alpha, 2 \alpha-\gamma}\left(-\lambda_{n} t^{\alpha}\right)\right)\right)(t)=t^{2 \alpha-1} E_{\alpha, 2 \alpha}\left(-\lambda_{n} t^{\alpha}\right),
$$
for $t>0$ and $0<\gamma<2 \alpha$ (e.g., formula (1.100) on p. 25 in \cite{5}), we have
\begin{equation}\label{eq2.7}
\left(J^{\gamma}\right)^{-1}\left(t^{2 \alpha-1} E_{\alpha, 2 \alpha}\left(-\lambda_{n} t^{\alpha}\right)\right)(t)=t^{2 \alpha-1-\gamma} E_{\alpha, 2 \alpha-\gamma}\left(-\lambda_{n} t^{\alpha}\right).
\end{equation}
Applying Lemma \ref{lem1} and \eqref{eq2.6}, we obtain
$$
\begin{aligned}
& \left(J^{\gamma}\right)^{-1}\left(t^{2 \alpha-1} E_{\alpha, 2 \alpha}\left(-\lambda_{n} t^{\alpha}\right) * p\right)(t)\left(\lambda_{n} P_{n} f, P_{n} \psi\right)_{X} \\
= & \left(t^{2 \alpha-1-\gamma} E_{\alpha, 2 \alpha-\gamma}\left(-\lambda_{n} t^{\alpha}\right) * p\right)(t)\left(\lambda_{n} P_{n} f, P_{n} \psi\right)_{X}.
\end{aligned}
$$
Therefore, by \eqref{eq2.6}, the proof of Lemma \ref{lem3} (i) is finished.

\textbf{Proof of Lemma \ref{lem3} (ii).}

Since (\ref{estimateML}) we have
$$
\begin{aligned}
& \left|\sum_{n=1}^{\infty}(t-s)^{2 \alpha-1-\gamma} E_{\alpha, 2 \alpha-\gamma}\left(-\lambda_{n}(t-s)^{\alpha}\right)\left(\lambda_{n} P_{n} f, \psi\right)_{X}\right| \\
& \leq C|t-s|^{2 \alpha-1-\gamma} \sum_{n=1}^{\infty}\left|\left(\lambda_{n} P_{n} f, \psi\right)_{X}\right|.
\end{aligned}
$$
Moreover,
$$
\sum_{n=1}^{\infty}\left|\left(\lambda_{n} P_{n} f, \psi\right)_{X}\right|=\sum_{n=1}^{\infty}\left|\left(P_{n} f, P_{n} A \psi\right)_{X}\right| \leq \sum_{n=1}^{\infty}\left\|P_{n} f\right\|_{X}\left\|P_{n} A \psi\right\|_{X}
$$
\begin{equation}\label{eq2.9}
\leq\left(\sum_{n=1}^{\infty}\left\|P_{n} f\right\|_{X}^{2}\right)^{\frac{1}{2}}\left(\sum_{n=1}^{\infty}\left\|P_{n} A \psi\right\|_{X}^{2}\right)^{\frac{1}{2}}=\|f\|_{X}\|A \psi\|_{X}.
\end{equation}
Thus, the proof of Lemma \ref{lem3} (ii) is complete.

\textbf{Proof of Lemma \ref{lem3} (iii).}

Setting
$$
R_{N}(t):=-\sum_{n=1}^{N}\left(t^{2 \alpha-1} E_{\alpha, 2 \alpha}\left(-\lambda_{n} t^{\alpha}\right) * p\right)(t)\left(\lambda_{n} P_{n} f, \psi\right)_{X}, \quad N \in \mathbb{N},
$$
by \eqref{eq2.7} we have
$$
\left(J^{\gamma}\right)^{-1} R_{N}(t)=-\sum_{n=1}^{N}\left(t^{2 \alpha-\gamma-1} E_{\alpha, 2 \alpha-\gamma}\left(-\lambda_{n} t^{\alpha}\right) * p\right)(t)\left(\lambda_{n} P_{n} f, \psi\right)_{X}.
$$
Therefore, the Young inequality, \eqref{estimateML} and \eqref{eq2.9} yield
$$
\begin{aligned}
& \left\|\left(J^{\gamma}\right)^{-1} R_{N}\right\|_{L^{2}(0, T)} \leq \sum_{n=1}^{N}\left\|t^{2 \alpha-\gamma-1} E_{\alpha, 2 \alpha-\gamma}\left(-\lambda_{n} t^{\alpha}\right)\right\|_{L^{1}(0, T)}\|p\|_{L^{2}(0, T)}\left|\left(\lambda_{n} P_{n} f, \psi\right)_{X}\right| \\
\leq & C \int_{0}^{T} t^{2 \alpha-\gamma-1} d t\|p\|_{L^{2}(0, T)}\|f\|_{X}\|A \psi\|_{X}.
\end{aligned}
$$
Hence, $R_{N}$ converges to $R$ as $N \rightarrow \infty$ in $H_{\gamma}(0, T)$. Thus, the proof of (iii) and so Lemma \ref{lem3} is completed.

\begin{lemma}\label{lem311}
    Let $\psi \in \mathcal{D}(A^2)$, $f \in X$, and $\alpha\in (0,1) \cup (1,2)$. Then,
\begin{align*}
& (u_{p,\alpha}(t), \, \psi)_X
= \sum_{n=1}^\infty \left( t^{\alpha-1}E_{\alpha,\alpha}(-\lambda_nt^{\alpha})\, *\,
p(t)\right)(P_nf, \, \psi)_X\\
=& \sum_{n=1}^\infty \left( \frac{t^{\alpha-1}}{\Gamma(\alpha)}\, *\, p(t)\right)
(P_nf, \, \psi)_X
- \sum_{n=1}^\infty \left( \frac{t^{2\alpha-1}}{\Gamma(2\alpha)}\,*\, p(t)\right)
(P_nf,\lambda_n\psi)_X\\
+& \sum_{n=1}^\infty (t^{3\alpha-1}E_{\alpha,3\alpha}(-\lambda_nt^{\alpha})\, *\, p(t))
(P_nf, \, \lambda_n^2\psi)_X
=: I_1(t) + I_2(t)+ S_{p, \alpha}(t),
\end{align*}
and
\begin{equation}\label{S}
\vert S_{p, \alpha}(t)\vert \leq C t^{3\alpha} \|p\|_{L^{\infty}(0, T)}\|f\|_{X}\|A^2 \psi\|_{X}, \,\, 0\leq t < 1.
\end{equation}
\end{lemma}
\begin{proof}The first part of the lemma follows directly from the obvious equality:
    \begin{align*}
& t^{\alpha-1}E_{\alpha,\alpha}(-\lambda_nt^{\alpha})
= t^{\alpha-1}\left( \frac{1}{\Gamma(\alpha)}
- \frac{\lambda_n t^{\alpha}}{\Gamma(2\alpha)}
+ \sum_{k=2}^\infty \frac{(-\lambda_nt^{\alpha})^k}{\Gamma(\alpha k + \alpha)}
\right)\\
=& \frac{t^{\alpha-1}}{\Gamma(\alpha)}
- \frac{\lambda_n t^{2\alpha-1}}{\Gamma(2\alpha)}
+ t^{3\alpha-1}\lambda_n^2E_{\alpha,3\alpha}(-\lambda_nt^{\alpha}).
\end{align*}
Estimate (\ref{estimateML}) implies (\ref{S}):
\[\vert S_{p, \alpha}(t)\vert \le C \|f\|_{X}\|A^2 \psi\|_{X}\int^t_0 (t-s)^{3\alpha-1}\vert p(s)\vert ds
\leq C t^{3\alpha} \|p\|_{L^{\infty}(0, T)}\|f\|_{X}\|A^2 \psi\|_{X}.
\]
\end{proof}

\section{Proof of Theorems \ref{thm1} and \ref{thm11}}\label{sec3}

First, we prove the following simple lemma.

\begin{lemma}\label{pp0}
    Let $p \in \mathcal{P}$, that is, there exist constants $C > 0$, $m \geq 0$, $a > 0$ with $a - m \geq \varepsilon > 0$, and $p_0 \neq 0$ such that
$$
\bigl| p(s) - p_0 s^{m} \bigr| \leq C s^{a} \quad \text{for all } s \in [0,T].
$$
Then the following asymptotic representation holds:
$$
p(s) = p_0 s^{m} + O\bigl(s^{m + \varepsilon}\bigr) \qquad \text{as } s \to 0^+.
$$
Moreover, this estimate is uniform on the interval $[0,1]$ whenever $T \geq 1$.
\end{lemma}

\begin{proof}
By assumption, there exists a remainder term $r(s)$ satisfying
$$
p(s) = p_0 s^{m} + r(s), \qquad |r(s)| \leq C s^{a} \quad \text{for all } s \in [0,T].
$$
Since $a \geq m + \varepsilon$, we have for every $s \in [0,T]$ the inequality
$$
s^{a} = s^{m + \varepsilon} \cdot s^{a - m - \varepsilon} \leq s^{m + \varepsilon} \cdot T^{a - m - \varepsilon},
$$
where the last step follows from $s \leq T$ and the non-negative exponent $a - m - \varepsilon \geq 0$. Therefore,
$$
|r(s)| \leq C \cdot T^{a - m - \varepsilon} \cdot s^{m + \varepsilon}.
$$
Setting
$$
C' := C \cdot T^{a - m - \varepsilon}
$$
(a constant that depends only on $p$, $T$ and the fixed $\varepsilon > 0$), we obtain
$$
|r(s)| \leq C' \, s^{m + \varepsilon} \quad \text{for all } 0 \leq s \leq T.
$$
This shows that
$$
p(s) = p_0 s^{m} + O\bigl(s^{m + \varepsilon}\bigr) \quad \text{as } s \to 0^+,
$$
and the $O$-estimate holds uniformly on the whole interval $[0,\min(1,T)]$.
\end{proof}

To prove Theorem \ref{thm1} we apply Lemma \ref{lem31}:
\[
\left(u_{p, \alpha}(t), \psi\right)_{X}=(f, \psi)_{X}\left(J^{\alpha} p\right)(t)+R_{p, \alpha}(t).
\]
Then, using $\psi \in \mathcal{D}(A)$, and  $p\in \mathcal{P}$ we have (see Lemma \ref{pp0})
$$
\vert R_{p, \alpha}(t)\vert \le  Ct^{2\alpha+m}, \quad 0 \le t <1,
$$
and
$$
 (J^{\alpha} p) (t)
= \frac{p_0}{\Gamma(\alpha)} \int^t_0 (t-s)^{\alpha-1}s^{m_1} ds
+ O\left( \int^t_0 (t-s)^{\alpha-1} s^{m+\varepsilon} ds \right)
$$
$$
= \frac{p_0\Gamma(m+1)}{\Gamma(\alpha+m+1)}t^{\alpha+m}
+ O(t^{\alpha+m+\varepsilon}), \quad 0\le t <1.
$$
Therefore,
$$
(u_{p,\alpha}(t), \psi)_X = p_1t^{\alpha+m} + O(t^{2\alpha+m})
+ O(t^{\alpha+m+\varepsilon})
= p_1t^{\alpha+m} + o(t^{\alpha+m}),
$$
where $p_1\ne 0$ is a constant.  Since $q\in \mathcal{P}$, similarly (see Lemma \ref{pp0} and note $\beta\in (0, 1)\cup (1,2)$),
$$
(u_{q,\beta}(t), \widetilde{\psi})_X = q_1t^{\beta+n} + o(t^{\beta+n}),
\quad 0\le t <1, \quad \text{with}\,\, q_1 \neq 0.
$$
By the conditions of the theorem,
$$
(u_{p,\alpha}(t), \, \psi)_X = (u_{q,\beta}(t), \, \widetilde{\psi})_X
\quad \text{for} \,\, 0\le t <1,                      
$$
we obtain
\begin{equation}\label{pq}
p_1t^{\alpha+m} + o(t^{\alpha+m}) = q_1t^{\beta+n} + o(t^{\beta+n}),
\quad 0\le t <1.                           
\end{equation}

Moreover, we have
\begin{lemma}\label{gamma1,2}
    Let $\gamma_1, \gamma_2 > 0$ and let $d_1, d_2 \ne 0$.
If
$$
d_1t^{\gamma_1} + o(t^{\gamma_1}) = d_2t^{\gamma_2} + o(t^{\gamma_2})
$$
as $t \downarrow 0$, then $\gamma_1 =\gamma_2$ and $d_1 = d_2$.
\end{lemma}
\textbf{Proof of Lemma \ref{gamma1,2}}.

Assume that $\gamma_1 > \gamma_2$.  Then,
$$
d_1t^{\gamma_1-\gamma_2} + o(t^{\gamma_1-\gamma_2}) = d_2 + o(1).
$$
Letting $t \downarrow 0$, we obtain $d_2=0$, which is a contradiction
in view of $d_2 \ne 0$.
Hence, $\gamma_1 \le \gamma_2$.  Similarly $\gamma_1 < \gamma_2$ yields
a contradiction.  Therefore, $\gamma_1 = \gamma_2$, and so
$(d_1-d_2)t^{\gamma_1} + o(t^{\gamma_1}) = 0$, that is,
$d_1-d_2 + o(1) = 0$.  Letting $t\downarrow 0$, we obtain $d_1 = d_2$.

Thus, the proof of the lemma is complete.

Applying Lemma to (\ref{pq}), we obtain
$$
\alpha+m = \beta + n, \quad \mbox{that is,} \quad
\alpha-\beta = n - m.                      
$$
By $m, n \in \mathbb{N}$, we obtain $n - m \in \mathbb{Z}$.  Since
both of $\alpha$ and $\beta$ belong either $(0,1)$ or $(1,2)$ implies $-1<\alpha-\beta < 1$. we have only a
possibility $\alpha- \beta = 0$.
Thus, the first part of Theorem \ref{thm1} is proved.

Next we will prove that $p=q$ in ( $0, T$ ) under the assumption $A=\widetilde{A}$ and $\alpha=\beta$ (as proved above). Setting $r:=p-q$ and $y:=u_{p, \alpha}-u_{q, \alpha}$, we have
$$
\partial_{t}^{\alpha} y(t)+A y=r(t) f, \quad y\in H_{\alpha}(0, T ; X),
$$
and
$$
(y(t), \psi)_{X}(t)=0, \quad 0<t<T .
$$

It suffices to prove that $r=0$ in ( $0, T$ ). Application of Lemma \ref{lem3} (i) yields
$$
\begin{aligned}
& 0=\left(\partial_{t}^{\alpha} y(t), \psi\right)_{X} \\
= & (f, \psi)_{X} r(t)-\sum_{n=1}^{\infty} \int_{0}^{t}(t-s)^{\alpha-1} E_{\alpha, \alpha}\left(-\lambda_{n}(t-s)^{\alpha}\right) r(s) d s\left(\lambda_{n} P_{n} f, \psi\right)_{X}.
\end{aligned}
$$
Since $(f, \psi)_{X} \neq 0$ and $r \in L^{\infty}(0, T)$, we obtain
$$
r(t)=\frac{1}{(f, \psi)_{X}} \sum_{n=1}^{\infty} \int_{0}^{t}(t-s)^{\alpha-1} E_{\alpha, \alpha}\left(-\lambda_{n}(t-s)^{\alpha}\right) r(s) d s\left(\lambda_{n} P_{n} f, \psi\right)_{X},
$$
and so by Lemma \ref{lem3} (ii) and \eqref{eq2.9}, we can derive
$$
|r(t)| \leq C \int_{0}^{t}(t-s)^{\alpha-1}|r(s)| d s, \quad 0<t<T.
$$

The generalized Gronwall inequality (see, e.g.,~\cite{3}, Chapter 4.1) implies that $r(t)=0$ for all $0<t<T$. This immediately yields the second part of the theorem. The proof of Theorem~\ref{thm1} is thus complete.

Now let us proceed to the proof of Theorem \ref{thm11}.
Since \(p \in \widetilde{\mathcal{P}}\), in view of Lemma \ref{lem311}, we can estimate and calculate:
$$
\vert S_{p, \alpha}(t)\vert 
\le Ct^{3\alpha+m},
$$
and (see Lemma \ref{pp0})
$$
I_1(t) = (f,\psi)_X \frac{p_0\Gamma(m+1)}{\Gamma(\alpha+m+1)}
t^{\alpha+m} + O(t^{\alpha+a})
=: p_1t^{\alpha+m} + O(t^{\alpha+a}),
$$
and
\begin{align*}
& I_2(t) = (f,A\psi)_X \left( \frac{t^{2\alpha-1}}{\Gamma(2\alpha)}\, *\,
p(t) \right)
= (f,A\psi)_X \frac{p_0\Gamma(m+1)}{\Gamma(2\alpha+m+1)}
t^{2\alpha+m} + O(t^{2\alpha+a})                           \\
=: & p_2t^{2\alpha+m} + O(t^{2\alpha+a}).
\end{align*}
We can similarly calculate $(u_{q,\beta}(t), \, \widetilde{\psi})_X$,  so that
we can obtain
\begin{align*}
& p_1t^{\alpha+m} + O(t^{\alpha+a}) + p_2t^{2\alpha+m}
+ O(t^{2\alpha+a}) + O(t^{3\alpha+m}) \\
= & q_1t^{\beta+n} + O(t^{\beta+b}) + q_2t^{2\beta+n}
+ O(t^{2\beta+b}) + O(t^{3\beta+n}).
\end{align*}
Since \(p, q \in \widetilde{\mathcal{P}}\) we have $a-m\geq 2$, and $b-n\geq 2$. Therefore, as $t \downarrow 0$ one has
\[
p_1t^{\alpha+m} +  p_2t^{2\alpha+m}
+ o(t^{2\alpha+m}) 
=  q_1t^{\beta+n} + q_2t^{2\beta+n}
+ o(t^{2\beta+n}).
\] 

Now we will apply the following lemma, which is easily proved by repeated application of Lemma  \ref{gamma1,2} (note that $p_1, q_1, p_2, q_2 \ne 0$).
\begin{lemma}\label{gamma1,2,3,4}
Let \(\gamma_1,\gamma_2 > 0\), \(\beta_1 > \gamma_1\), \(\beta_2 > \gamma_2\), and \(a_1,a_2,b_1,b_2 \neq 0\). If
\[
a_1 t^{\gamma_1} + b_1 t^{\beta_1} + o(t^{\beta_1}) = a_2 t^{\gamma_2} + b_2 t^{\beta_2} + o(t^{\beta_2})
\quad \text{as} \quad t \downarrow 0,
\]
then necessarily
\[
\gamma_1 = \gamma_2, \quad \beta_1 = \beta_2, \quad a_1 = a_2, \quad b_1 = b_2.
\]
\end{lemma}
Thus $p_1 = q_1$, $p_2 = q_2$ and $\alpha +m = \beta +n$, $2\alpha +m = 2\beta +n$. Recall $\alpha, \beta \in (0,1)\cup (1,2)$. As we have established above, if both $\alpha, \beta$ belong to either $(0,1)$ or $(1,2)$, then from equality $\alpha +m = \beta +n$ we obtain $\alpha = \beta$.

Now, without loss of generality, let us assume that $\alpha \in (0,1)$ and $\beta \in (1,2)$. Then equality $\alpha +m = \beta +n$ implies $m+1 =n$, but this contradicts equality $2\alpha +m = 2\beta +n$. Consequently, both $\alpha, \beta$  reside in the same set among $(0,1)$ or $(1,2)$, and moreover, $\alpha = \beta$. Thus Theorem \ref{thm11} is proved.

\section{The case $\alpha\in (0,1)$ and 
$p\in \mathcal{P}_0$}\label{sec4}
\begin{proof}[Proof of Theorem~\ref{thm5}]
Assume that $p \in \mathcal{P}_0$, $p \not\equiv 0$ in $(0,T)$, and 
\[
\partial_t^\alpha \bigl(u_{p,\alpha}, \psi\bigr)_X \in H_\theta(0,T).
\]
By Lemma~\ref{lem3}(i), for $p \in L^2(0,T)$, we have
\[
(f,\psi)_X\, p(t)
= \partial_t^\alpha \bigl(u_{p,\alpha}(t), \psi\bigr)_X 
- \partial_t^\alpha R_{p,\alpha}(t),
\quad 0<t<T.
\]
Choosing $\gamma < 2\alpha$ arbitrarily close to $2\alpha$ in Lemma~\ref{lem3}(iii), 
we obtain $\partial_t^\alpha R_{p,\alpha} \in H_{\alpha-\varepsilon}(0,T)$ for all sufficiently small $\varepsilon>0$. Hence,
\[
(f,\psi)_X\, p \in H_{\min\{\theta,\alpha-\varepsilon\}}(0,T).
\]

Since $(f,\psi)_X \neq 0$, it follows that 
\[
p \in H_{\min\{\theta,\alpha-\varepsilon\}}(0,T).
\]
Because $\alpha > \tfrac{1}{2}$, we can choose $\varepsilon>0$ sufficiently small so that
\[
\gamma := \min\{\theta,\alpha-\varepsilon\} > \tfrac{1}{2}.
\]
Then $p \in H_\gamma(0,T)$ implies $p \in C[0,T]$, and in particular $p(0)=0$. 
This contradicts the assumption $p \in \mathcal{P}_0$. Therefore, $p \equiv 0$ in $(0,T)$.
\end{proof}
\begin{proof}[Proof of Theorem \ref{thm6}] Let $\psi \in \mathcal{D}(A)$,  $(f, \psi)_{X} \neq 0$ and $\lim\limits_{t\to +0} u_{p, \alpha}(t) =0$, $\lim\limits_{t\to +0} \partial_{t}^{\alpha}\left(u_{p, \alpha}, \psi\right)_{X}(t) =0$. Multiplying equation (\ref{sol}) by $\psi$ we get
\[
\left(\partial_{t}^{\alpha}\, u_{p, \alpha}, \psi\right)_{X} = - ( A u_{p, \alpha}, \psi)_{X} + p(t) ( f, \psi)_X.
\]
By Lemma \ref{lem1.1} and conditions of the theorem
\[
\lim_{t\to + 0} \left(\partial_{t}^{\alpha}\, u_{p, \alpha}, \psi\right)_{X} = \lim_{t\to + 0}  \partial_{t}^{\alpha}\,( u_{p, \alpha}, \psi)_{X}  =0.
\]
Since $A$ is self-adjoint, then again by the conditions of the theorem
\[
\lim_{t\to +0} ( A u_{p, \alpha}, \psi)_{X} = \lim_{t\to + 0} ( u_{p, \alpha}, A \psi)_{X} =0.
\]
Then, $p(0)=0$ because $(f, \psi)_{X} \neq 0$. This implies that $p \in C[0, T]$ and $p(0)=0$, which is a contradiction by $p \in \mathcal{P}_0$. Therefore $p \equiv 0$ in $(0, T)$, and the proof of Theorem \ref{thm6} is complete.

\end{proof}

\begin{proof}[Proof of Corollary \ref{cor2}]
    From the conditions of the corollary and from equation (\ref{sol}) it follows that, $p \in H_{\gamma}(0, T)$ with $\gamma>\frac{1}{2}$. On the other hand, we can directly verify that (note that in the integration domain one has $|t-s|^{1+2 \gamma} = (t-s)^{1+2 \gamma}$)
$$
\int_{0}^{a}\left(\int_{a}^{T} \frac{1}{|t-s|^{1+2 \gamma}} d t\right) d s=\frac{1}{2 \gamma} \int_{0}^{a}\left((a-s)^{-2 \gamma}-(T-s)^{-2 \gamma}\right) d s
$$
does not exist for $\gamma>\frac{1}{2}$ since $(T-s)> (a-s)$. Therefore, piecewise constant $p \in H_{\gamma}(0, T)$ with $\gamma>\frac{1}{2}$ yields that $p$ is a constant function in $(0, T)$. Otherwise $p \notin H_{\gamma}(0, T)$. Moreover, $p(0)=0$ by $p \in H_{\gamma}(0, T)$ with $\gamma>\frac{1}{2}$. Thus $p \equiv 0$ in ( $0, T$ ).
\end{proof}

\begin{proof}[Proof of Proposition \ref{prop1}]
    In terms of $(f, \psi)_{X}=0$, it follows from Lemma \ref{lem31}  and  Lemma \ref{lem3} (iii) that
$$
\left(u_{p, \alpha}, \psi\right)_{X}=R_{p, \alpha} \in H_{\gamma}(0, T),
$$
where $\gamma \in(0,2 \alpha)$ is arbitrary. Thus the proof of Proposition \ref{prop1} is complete.
\end{proof}

\section{Inverse Problem 2}\label{sec5}

In this section, we will investigate Inverse Problem 2.

Interest in the study of inverse source identification problems is motivated by their importance in various fields, including mechanics, seismology, medical tomography, and geophysics. The theory and applications of such inverse problems for integer- and fractional-order partial differential equations have been extensively investigated in the literature (see, e.g., the monograph by Kabanikhin~\cite{Kabanikhin} and the survey article by Liu, Li, and Yamamoto~\cite{LiYamamoto3}).

Inverse problems of determining the right-hand side in various subdiffusion equations have been studied by many authors. The corresponding methods of analysis depend on whether the spatial source term $f$ or the temporal component $p(t)$ is unknown (see, e.g.,~\cite{LiYamamoto3}).
When $f$ is unknown, the study was conducted in two directions: $p(t) \equiv 1$  and $p(t) \not\equiv 1$ (see, for example, \cite{AM2020},  \cite{Kirane}, \cite{LiYamamoto3}, \cite{Ruzhansky}, and the literature therein). In both cases, the method of separation of variables was used. In the case of $p(t) \not\equiv 1$, preserving the sign of this function plays a key role (see e.g., \cite{ASh2024} and the literature therein).

Investigating the case of $p(t)$ is unknown and $f$ is known is much more difficult (see \cite{ASh2022}, \cite{LiYamamoto3}, and the literature therein). Here, the definition of $p(t)$ is usually reduced to an integral equation (the Volterra or Fredholm integral equations), and in some cases, the method of compressed mappings is used to solve it (see, for example, \cite{LiYamamoto3}).

In this paper, the solution to Inverse Problem 2 is easily derived from Lemma \ref{lem3}. Let us present a proof of Theorem \ref{thm3}.

\begin{proof} If we put $\gamma = \alpha$ in statement (i) of Lemma \ref{lem3}, we get
\begin{equation}\label{Volterra}
(f, \psi)_{X} \, p(t) = \partial_t^\alpha \Psi (t) + \int_0^t p(s) (t-s)^{\alpha-1} \sum_{n=1}^{\infty} E_{\alpha,  \alpha}(-\lambda_{n} (t-s)^{\alpha}) (P_{n} f, P_n A\psi)_{X} ds.
\end{equation}
By virtue of estimate (\ref{estimateML}), the following estimate holds.:
\begin{equation*}
\sup _{n \in \mathbb{N}, 0<s<t}\left|E_{\alpha,  \alpha}\left(-\lambda_{n}(t-s)^{\alpha}\right)\right|=: C_{\alpha}<\infty. 
\end{equation*}
Therefore, the following series is uniformly convergent on $s\in [0, t]$ (see the proof of Lemma \ref{lem3} (ii)):
\[
\sum_{n=1}^{\infty} \big|E_{\alpha,  \alpha}(-\lambda_{n} (t-s)^{\alpha}) (P_{n} f, P_n A\psi)_{X}\big|\leq C_\alpha \|f\|_{X}\|A \psi\|_{X}.
\]
Hence the Volterra equation (\ref{Volterra}) has the unique solution $p(t)$ and moreovere one has (note $(f, \psi)_{X} \neq 0$)
\begin{equation*}
||p||_{C[0, T]} \leq C_{f, \psi}\left[||\partial_t^\alpha \Psi||_{C[0, T]} + C_\alpha ||f||_X ||A\psi||_X \int_0^t (t-s)^{\alpha-1} |p(s)| ds\right],\,\, C_{f, \psi} = \frac{1}{|(f, \psi)_{X}|}.
\end{equation*}
Now the generalized Gronwall inequality (e.g., \cite{4}, \cite{8}) yields 
\begin{equation*}
||p||_{C[0, T]} \leq C_{f, \psi}||\partial_t^\alpha \Psi||_{C[0, T]} E_{\alpha} (C_{f, \psi} C_\alpha ||f||_X ||A\psi||_X \Gamma(\alpha) T^\alpha).
\end{equation*}

If we again substitute $\gamma = \alpha$ into statement (i) of Lemma \ref{lem3}, and use the reasoning above, we will have
\[
||\partial_t^\alpha \Psi||_{C[0, T]} \leq  |(f, \psi)_{X}| \, || p(t)||_{C[0, T]} +  C_\alpha ||f||_X ||A\psi||_X \int_0^t (t-s)^{\alpha-1} |p(s)| ds \leq C_{f, \psi, \alpha} || p(t)||_{C[0, T]},
\]
where the constant $C_{f, \psi, \alpha}$ depends on $f, \psi, $ and $\alpha$.

Combining the last two estimates, we obtain (\ref{stability}).
    
\end{proof}

\begin{remark} If the conditions of Theorem \ref{thm6} are met, the statement of this theorem also follows from relation (\ref{Volterra}).
\end{remark}

\begin{remark} In Theorem \ref{thm3}, the unknown function $p(t)\in C[0, T]$ is uniquely restored without the additional condition $p(t) \in \mathcal{P}$. On the other hand, Theorem \ref{thm1} guarantees the uniqueness of the solution to Inverse Problem 2 - even with an unknown $\alpha$ and without any additional smoothness requirements for the function $\Psi(t)$. Theorem \ref{thm5} essentially states that if $\alpha \in (1/2, 1)$, then - even in the case where $\Psi(t) \neq 0$, but is merely sufficiently smooth (specifically: $\Psi(t) \in H_\theta(0, T; X)$ for $1/2 < \theta < \alpha$) - Inverse Problem 2 possesses a unique solution $p(t) \in \mathcal{P}_0$. Theorem \ref{thm3} showed that if $\theta = \alpha$, then one can also assert the existence of a solution to the Inverse problem 2.
\end{remark}

\section{Concluding remarks}\label{sec6}
This work addresses the inverse problem of simultaneously determining both the exponent $\alpha$ and the multiplier $p(t)$ in the right-hand side of the equation in problem~(\ref{sol}). The quantity $\left(u_{p,\alpha}, \psi\right)_{X}(t)$, for $0<t<T$, is taken as the overdetermination condition, where $\psi \in D(A)$ is an arbitrary element.

The uniqueness of the parameter $\alpha$ is established independently of the choice of the elliptic operator $A$ and the source term. Moreover, the uniqueness of the function $p(t)$ holds not only in the case where the additional condition vanishes, i.e.,
\[
(u_{p,\alpha}, \psi)_{X}(t)=0,
\]
but also when the function $(u_{p,\alpha}, \psi)_{X}(t)$ is sufficiently smooth. The proofs of all results rely on the key Lemmas~\ref{lem3} and~\ref{lem31}.

As a possible generalization of the obtained results, we note the following.

It is plausible that the uniqueness property would still hold if, instead of the operator $A$, one considers the Laplace operator with Dirichlet boundary conditions in a certain $N$-dimensional domain~$\Omega$, and prescribes the overdetermination condition in the form $u_{p,\alpha}(x_0,t)$, $x_0\in \Omega$.

However, this generalization remains a subject for future research.

{\bf Acknowledgements.}
Masahiro Yamamoto was supported by 
Grant-in-Aid for Challenging Research (Pioneering) 21K18142 of 
Japan Society for the Promotion of Science.

\end{document}